\documentclass[12pt]{article}
\usepackage{authblk}
\usepackage{graphicx} 
\usepackage{amsmath}
\usepackage{amsfonts}
\usepackage{amssymb}
\usepackage{amsthm}
\usepackage{float}
\usepackage{geometry}	
\usepackage{mathtools}
\usepackage{fontenc}
\usepackage[dvipsnames]{xcolor}
\usepackage[pdftex, backref=page]{hyperref} 
\renewcommand*{\backref}[1]{}
\renewcommand*{\backrefalt}[4]{%
	\ifcase #1 (Not cited.)%
	\or        (Cited on page~#2.)%
	\else      (Cited on pages~#2.)%
	\fi}

\usepackage[initials]{amsrefs}
\usepackage[english]{babel}

\usepackage{hyperref}
\usepackage{amssymb}
\usepackage{graphicx}
\usepackage{mathrsfs}
\usepackage{upref,amsthm,amsxtra,exscale}
\usepackage{cite}

\usepackage{caption}
\usepackage[shortlabels]{enumitem}
\usepackage{esint}
\setlist[description]{leftmargin=\parindent,labelindent=\parindent,itemsep=1pt,parsep=0pt,topsep=0pt}

\newtheorem{theorem}{Theorem}[section]
\newtheorem{corollary}[theorem]{Corollary}
\newtheorem{lemma}[theorem]{Lemma}
\newtheorem{proposition}[theorem]{Proposition}

\numberwithin{equation}{section}
\newtheorem*{subject}{2020 Mathematics Subject Classification}
\newtheorem*{keywords}{Keywords}

\theoremstyle{definition}
\newtheorem{remark}[theorem]{Remark}

\usepackage{xcolor}

\newcommand{\epsi}{\mbox{$\varepsilon$}}
\newcommand{\eps}{\mathchoice{\epsi}{\epsi}
	{\mbox{\scriptsize\epsi}}{\mbox{\tiny\epsi}}}

\newcommand{\NN}{\mathbb{N}}

\newcommand{\RR}{\mathbb{R}}

\DeclareMathOperator{\sign}{sign}
\title{Boundedness of Solutions to Nonlocal Systems}
\date{}
\author[1]{Serena Dipierro\thanks{serena.dipierro@uwa.edu.au}}
\author[2]{João Gonçalves da Silva\thanks{joao.goncalvesdasilva@adelaide.edu.au}}
\author[2]{Giorgio Poggesi\thanks{giorgio.poggesi@adelaide.edu.au}}
\author[1]{Enrico Valdinoci\thanks{enrico.valdinoci@uwa.edu.au}}

\affil[1]{ Department of Mathematics and Statistics, The University of Western Australia, 35 Stirling Highway, Perth, WA 6009, Australia}

\affil[2]{School of Mathematical Sciences, Adelaide University, Adelaide SA 5005, Australia.}

\begin{document}
	
	\maketitle
	\begin{abstract}
Under mild mass or energy assumptions, we prove uniform
$L^\infty$ bounds for all distributional solutions of a multi-dimensional system of nonlocal semilinear equations.

The setting comprises
Ginzburg-Landau systems and Allen-Cahn equations
and the results obtained are new even for scalar equations.
	\end{abstract}
	
		\begin{keywords}
		Nonlocal Equations, Ginzburg-Landau systems, De Giorgi Conjecture, Boundedness of Solutions, Kato Inequality
	\end{keywords}
	\begin{subject}
		45H05; 35B65; 35B45; 35R11; 35B06
	\end{subject}
	\section{Introduction}
	\subsection{Mathematical Setting and Main Results}
	Given $N,M\in\mathbb{N}$, we study the boundedness of distributional solutions, $u:\RR^N\rightarrow \RR^M$ to the problem 
	\begin{equation}\label{Main Problem}
		\begin{cases}
			\mathcal{L}_{K} u =  -uP_{n}'(|u|^2)\, \text{ in }\, \mathbb{R}^N,\\
			u\in L^{4n-1}_{\text{loc}}(\RR^N; \RR^M),
		\end{cases}
	\end{equation}
	where $P_n$ is a polynomial of the form 
	$$
	P_n(t) = \frac{1}{2}\prod_{i=1}^{n} (t-k_i)^2, \,\, \text{with}\,\,n\in\NN\,\,\text{and}\,\,0<k_1<...<k_n<\infty,
	$$
	and $\mathcal{L}_K u$ denotes the nonlocal operator (formally) given by 
		\begin{equation}\label{def:nonlocal operator}
			\mathcal{L}_{K} u(x)\coloneqq P.V. \int_{\RR^N} (u(x)-u(y)) K(x-y)\,dy = \lim_{\eps\to 0^+} \int_{\RR^N\setminus B_{\eps}(x)} (u(x)-u(y)) K(x-y)\,dy,
		\end{equation}
		where the kernel $K:\RR^N\rightarrow\RR$ is a measurable, even function satisfying:
		\begin{equation}\label{eq:comparison of k with FL}
			0\leq K(x)|x|^{N+2s}\leq\Lambda \quad \text{for a.e.}\,\, x\in \RR^N,
		\end{equation}
		for some $\Lambda> 0$ and $s\in(0,1)$. Here above, $P.V.$ stands for ``in the principal value sense''.
		In what follows, the parameters $\Lambda$ and $s\in(0,1)$ are fixed. 

	The main example of an operator that fits the assumptions set for $\mathcal{L}_K$ is the fractional Laplacian, i.e., $\mathcal{L}_K = (-\Delta)^s$ for some $s\in(0,1)$. In this case, the kernel $K$ is given by~$K(x)= C_{N,s}|x|^{-N-2s}$,
where $s\in (0,1)$ and $C_{N,s}$ is a suitable normalization constant. We refer the interested reader to \cite{MR2944369} and the
references therein for the basic properties of the fractional Laplacian.
	
	In this setting, our main result\footnote{Theorem~\ref{thm:boundedness of solutions} answers a question posed to us by one of the referees of~\cite{MR4968697} for scalar Allen-Cahn equations.
	The answer provided in Theorem~\ref{thm:boundedness of solutions} comprises also more general
Ginzburg-Landau systems.
	We are indebted to this anonymous referee for the stimulating question.} is the following:
	\begin{theorem}\label{thm:boundedness of solutions}
		Let~$K:\RR^N\rightarrow \RR$ be a measurable, even function satisfying \eqref{eq:comparison of k with FL}.
		Let~$u\in L^{4n-1}_{\text{loc}}(\RR^N;\RR^M)$ be a distributional solution of \eqref{Main Problem}, such that 
		\begin{equation}\label{eq:integrability assumption on sol}
			\int_{\RR^N}\frac{|u(x)|}{1+|x|^{N+2s}}\,dx<\infty.
		\end{equation}
		Assume either that\\
	    (i) There exist $x_0\in \RR^N$, $R_0$, $C>0$ and $h< N+6s$ such that, for all $R>R_0$,
			\begin{equation}\label{assumption 1 on sol}
				\iint_{T_{B_R(x_0)}}|u(x)-u(y)|^2\, K(x-y)\,dy\,dx \leq C R^h,
			\end{equation}
			where $T_{B_R(x_0)}=\left(B_{R}(x_0)\times \RR^N\right)\cup \left(\RR^N\times B_{R}(x_0)\right)$,\\
or that\\
		(ii) There exist $x_0\in \RR^N$, $R_0$, $C>0$ and $h< N+2s$ such that, for all $R>R_0$,
		\begin{equation}\label{assumption 2 on sol}
			\int_{B_{R}(x_0)}\left(|u(x)|-\sqrt{k_n}\right)^+\,dx\leq C R^{h}.
		\end{equation} 
		Then, $u\in L^{\infty}(\RR^N;\RR^M)$ and 
		\begin{equation}\label{eq:bound on solution}
			|u|\leq \sqrt{M k_n}.
		\end{equation}
	\end{theorem}
	Our proof of Theorem~\ref{thm:boundedness of solutions} is inspired by the classical proof of \cite[Proposition $2.2$]{MR1659256} for the local case, which relies on combining Kato's inequality (see, e.g., \cite[Lemma~A]{MR0333833}) with a mathematical induction argument which leverages integration by parts and the fact that the Laplacian is a local operator. Hence, adapting the proof of \cite[Proposition $2.2$]{MR1659256} presents a few obstacles. Namely, to our knowledge, under the assumptions of Theorem~\ref{thm:boundedness of solutions}, there are no available analogues in the literature of Kato's inequality for the operator~$\mathcal{L}_K$ or an integration by parts formula for $\mathcal{L}_K$, in the generality addressed in this paper. Additionally, the fact that the operator~$\mathcal{L}_K$ is nonlocal forces one to modify the mathematical induction argument in the proof of~\cite[Proposition~$2.2$]{MR1659256}. 
	
	To deal with these difficulties, first we obtain an integration by parts formula for the operator~$\mathcal{L}_K$, which holds for smooth enough functions. This is achieved in Lemma~\ref{lem:integration by parts for nonlocal operator}. Then, we provide a generalisation of Kato's inequality (see \cite[Lemma~A]{MR0333833}
	for the classical version) for the operator~$\mathcal{L}_K$. This is the content of Proposition~\ref{prop:Nonlocal Kato inequality}.\\
	To overcome the second obstacle, we assume that either \eqref{assumption 1 on sol} or \eqref{assumption 2 on sol} hold and modify the induction argument of the proof of \cite[Proposition $2.2$]{MR1659256}.
	
	\begin{remark}[On the assumptions in Theorem~\ref{thm:boundedness of solutions}]
		{\rm
			\begin{enumerate}[(i)]
				\item The assumption $u\in L^{4n-1}_{\text{loc}}(\RR^N,\RR^M)$ in Theorem~\ref{thm:boundedness of solutions} is in place solely to ensure $uP_{n}'(|u|^2)\in L^1_{\text{loc}}(\RR^N;\RR^M)$. 
				Indeed, since $P_{n}$ is a polynomial of degree $2n$, we can write
				\begin{equation*}
					uP_{n}'(|u|^2)= u \sum_{j=1}^{2n-1} c_j |u|^{2j},
				\end{equation*}
				for some coefficients $c_1,...,c_{2n-1}\in \RR$.
				Since $u\in L^{4n-1}_{\text{loc}}(\RR^N,\RR^M)$, it follows that $u P_{n}'(|u|^2)\in L^{1}_{\text{loc}}(\RR^N)$.\newline 
				Assumption \eqref{eq:integrability assumption on sol} is set so that the definition of distributional solution makes sense (see~\eqref{def:distributional solution} and (iii) in Remark \ref{rem:consequences of upper bounds} below).
				
				\item At a first glance, one might conjecture that since assumption \eqref{eq:integrability assumption on sol} is in place, maybe the assumptions \eqref{assumption 1 on sol}-\eqref{assumption 2 on sol} are not required. In fact, assumption \eqref{eq:integrability assumption on sol} implies \eqref{assumption 2 on sol} with $h =N+2s$ and any $x_0\in \RR^N$ and $R>|x_0|$ and some constant $C$. This, however, is not enough to make the aforementioned mathematical induction argument in the proof of Theorem~\ref{thm:boundedness of solutions} work. See Remark \ref{rem:necessity for aditional assumptions} for additional details.
				
				\item Regarding the validity of \eqref{assumption 1 on sol}, we note that if the solution $u$ has finite energy, namely
				$$
				\iint_{\RR^{N}\times \RR^N}|u(x)-u(y)|^2\, K(x-y)\,dy\,dx<\infty,
				$$
				then \eqref{assumption 1 on sol} holds for any $x_0\in\RR^N$, $R>0$ and $h=0$.\\
				As for the validity of assumption \eqref{assumption 2 on sol}, we note that if there exist $x_0\in \RR^N$, $R_0$, $\tilde{C}>0$ and $h< N+4s$ such that, for all $R>R_0$,
				\begin{equation*}
					\int_{B_{R}(x_0)}\left(\prod_{i=1}^{n}\left(|u|^2-k_i\right)^+\right)^2\,dx\leq C R^{h},
				\end{equation*}
				then \eqref{assumption 2 on sol} holds. Indeed, leveraging the assumption that  $0<k_1<...<k_n<\infty$ and the Cauchy-Schwartz inequality, we see that 
				\eqref{assumption 2 on sol} holds:
				\begin{equation*}
					\begin{split}
						&\int_{B_{R}(x_0)}\left(|u(x)|-\sqrt{k_n}\right)^+\,dx\\
						&\leq \left(\prod_{i=1}^{n-1}(\sqrt{k_n}-\sqrt{k_{i}})^+\right)^{-1}\int_{B_{R}(x_0)}\prod_{i=1}^{n}\left(|u|^2-k_i\right)^+\,dx\\
						&\leq  \left(\prod_{i=1}^{n-1}(\sqrt{k_n}-\sqrt{k_{i}})^+\right)^{-1}\left|B_{R}(x_0)\right|^{1/2}\left(\int_{B_{R}(x_0)}
						\left(\prod_{i=1}^{n}\left(|u|^2-k_i\right)^+\right)^2\,dx\right)^{1/2}\\
						&\leq C R^{\frac{h+N}{2}},
					\end{split}
				\end{equation*}
				for some positive constant $C$. Here, we used the standard convention that $\prod\limits_{i\in \varnothing} a_i =1$
				
				In particular, if 
				$$
				\int_{\RR^N} P(|u(x)|^2)\,dx<\infty,
				$$
				then \eqref{assumption 2 on sol} holds.
			\end{enumerate}
		}
	\end{remark}
	
	\begin{proposition}\label{prop: General Nonlocal Kato inequality}
		Let~$K:\RR^N\rightarrow \RR$ be a measurable, even function satisfying \eqref{eq:comparison of k with FL}, and let $\Phi:\RR\rightarrow \RR$ be a Lipschitz convex function.
		Then, for every $u\in L^{1}_{\text{loc}}(\RR^N)$ satisfying \eqref{eq:integrability assumption on sol}, and such that $\mathcal{L}_K u =f \in L^{1}_{\text{loc}}(\RR^N)$ (in the distributional sense), we have, for every nonnegative\footnote{Here and in what follows, $\mathcal{D}(\RR^N)$ denotes the set of compactly supported, smooth functions in $\RR^N$.} $\varphi\in \mathcal{D}(\RR^N)$, that
		\begin{equation}\label{eq:general Kato}
			\int_{\RR^N}\Phi(u(x))\mathcal{L}_K\varphi(x)\,dx\leq \int_{\RR^N}\Phi'(u(x))\varphi(x)\mathcal{L}_{K}u(x)\,dx. 
		\end{equation}
	\end{proposition}
	
	As a direct consequence of this result we have the following Kato inequality for the operator~$\mathcal{L}_K$: 
	\begin{proposition}\label{prop:Nonlocal Kato inequality}
		Let~$K:\RR^N\rightarrow \RR$ be a measurable, even function satisfying \eqref{eq:comparison of k with FL}.
		For every $u\in L^{1}_{\text{loc}}(\RR^N)$ satisfying \eqref{eq:integrability assumption on sol}, and such that $\mathcal{L}_K u  = f\in L^{1}_{\text{loc}}(\RR^N)$ (in the distributional sense), we have, for every nonnegative $\varphi\in\mathcal{D}(\RR^N)$, that
		\begin{equation}\label{eq:Kato for absolute value}
			\int_{\RR^N}|u(x)|\mathcal{L}_K\varphi(x)\,dx\leq \int_{\RR^N}\sign(u(x))\varphi(x)f(x)\,dx,
		\end{equation}
		where 
		$$
		\sign u(x): = \begin{cases}
			u(x)/|u(x)|,&\,\text{if}\, u(x)\neq 0\\
			0,&\,\,\text{otherwise},
		\end{cases}
		$$ is the sign function. In the inequality~\eqref{eq:Kato for absolute value}, we can replace $|u(x)|$ and $\sign(u(x))$ with their positive or negative parts, $u^{\pm}(x)=\max\{\pm u(x),0\}$, $\sign^{\pm}(u(x))=\max\{\pm\sign(u(x)),0\}$, respectively. 
	\end{proposition}
		Our proof of Proposition~\ref{prop: General Nonlocal Kato inequality} result follows the outline of the proof of the original Kato inequality in \cite[Lemma~A]{MR0333833}. In short, in a first step, we prove the result for functions $u\in C^{\infty}(\RR^N)$ that satisfy \eqref{eq:integrability assumption on sol} which, as shown in Lemma~\ref{lem:behaviour of the operator in D} satisfy the condition $\mathcal{L}_K u \in L^{1}_{\text{loc}}(\RR^N)$. At a second stage, we perform a density argument to conclude the result.
		
		 We note that in \cite[Proposition~$1.7$]{MR4893129} an analogue of Proposition~\ref{prop: General Nonlocal Kato inequality} is established for the fractional $p-$Laplacian, $\left(-\Delta\right)_{p}^s$, under assumptions that are slightly stronger than those of Proposition~\ref{prop: General Nonlocal Kato inequality} here. More precisely, \cite[Proposition~$1.7$]{MR4893129} is established under the assumptions that the function $\Phi$ (as in the statement of Proposition~\ref{prop: General Nonlocal Kato inequality}) is convex, Lipschitz, of class $C^{2}$, satisfies $\Phi(0)=0$ and either $\Phi''\in L^{\infty}(\RR)$ or $(-\Delta)_{p}^{s} u$ does not change sign.
		As a consequence, a version of Proposition~\ref{prop:Nonlocal Kato inequality} for the fractional $p-$Laplacian is obtained in \cite[Proposition $2.2$]{MR4893129}, by a combination of an approximation argument and \cite[Proposition~$1.7$]{MR4893129}.
	
	As a direct consequence of Theorem~\ref{thm:boundedness of solutions} and \cite[Theorem $1.2$]{MR4968697} we obtain the following rigidity result.
	\begin{corollary}\label{cor:fractional DeGiorgi in 2D}
		Let~$u\in C^{2}(\RR^2)$ be a solution of \eqref{Main Problem}, with $N=2$ and $M=1$, such that 
		\begin{equation}\label{eq:assumption 1 on frac gen AC}
			\int_{\RR^2}\frac{|u(x)|}{1+|x|^{2+2s}}\,dx<\infty.
		\end{equation}
		Assume either that\\
		(i) There exist $x_0\in \RR^2$, $R_0$, $C>0$ and $h< 2+6s$ such that, for all $R>R_0$,
		\begin{equation}\label{assumption 2 on frac gen AC}
			\iint_{T_{B_R(x_0)}}\frac{|u(x)-u(y)|^2}{|x-y|^{2+2s}}\,\,dy\,dx \leq C R^h,
		\end{equation}
		where $T_{B_R(x_0)}=\left(B_{R}(x_0)\times \RR^2\right)\cup \left(\RR^2\times B_{R}(x_0)\right)$.\\
		Or that\\
		(ii) There exist $x_0\in \RR^2$, $R_0$, $C>0$ and $h< 2+2s$ such that, for all $R>R_0$,
		\begin{equation}\label{assumption 3 on frac gen AC}
			\int_{B_{R}(x_0)}\left(|u(x)|-\sqrt{k_n}\right)^+\,dx\leq C R^{h}.
		\end{equation} 
		
		Then, $u$ is one-dimensional, i.e., there exist~$u_0:\mathbb{R}\to \mathbb{R}$ and~$\omega \in\mathbb{S}^1$ such that~$u(x) = u_0(\omega\cdot x)$.
	\end{corollary}
	\subsection{Organisation of the paper}
	The rest of the paper is organised as follows.
	In Section \ref{Section 2}, we present the definition of distributional solution we are working with and the functional spaces involved, and we obtain a series of lemmas that will be crucial to proving Theorem~\ref{thm:boundedness of solutions} and Proposition~\ref{prop: General Nonlocal Kato inequality}. 
	Section \ref{Section 3} is dedicated to proving Proposition~\ref{prop: General Nonlocal Kato inequality}. 
	In Section \ref{Section 4}, we prove Theorem~\ref{thm:boundedness of solutions}. In Section~\ref{sec:proof of De Giorgi} we prove Corollary~\ref{cor:fractional DeGiorgi in 2D}. We conclude the paper with Appendix \ref{Appendix}, where we collect some auxiliary and technical results, as well as the proofs of some simple lemmata.
	%%%%%%%%%%%%%%%%%%%%%%%%%%%%%%%%%%%%%%%%%%%%%%%%%%%%%%%%%%%%%%%%%%%%%%%%%%%%%%%%%%%%%%%%%%%%%%%%%%%%%%%%%%%%%%%%%%%%%%%%%%%%%%%%%%%%	 
	\section{The operator $\mathcal{L}_K$: distributional solutions and basic properties}\label{Section 2}
	
	This section is devoted to recalling the necessary background results needed for the proofs of Theorem~\ref{thm:boundedness of solutions} and Proposition~\ref{prop: General Nonlocal Kato inequality}. 
	
	We start by introducing the functional spaces used throughout the paper. 
	As previously mentioned, the set of smooth (i.e., infinitely differentiable), compactly supported functions from $\RR^N$ to $\RR$ will be denoted by $\mathcal{D}(\RR^N)$, and the set of twice continuously differentiable functions will be denoted by $C^{2}_{\text{loc}}(\RR^N)$.  
		Given $s\in (0,1)$, $N,M\in\mathbb{N}$, we introduce the space $L^{1}_{s}(\RR^N;\RR^{M})$ defined as
		\begin{equation}\label{def:L1s}
			L^{1}_{s}(\RR^N;\RR^{M}):= \left\{u\in L^{1}_{\text{loc}}(\RR^N; \RR^M):\, \int_{\RR^N}\frac{|u(x)|}{1+|x|^{N+2s}}\,dx<\infty\right\},
		\end{equation}
		where by $L^{1}_{\text{loc}}(\RR^N;\RR^{M})$ we mean the set of measurable functions defined in $\RR^N$ taking values $\RR^M$ that are locally integrable. In the case where $M=1$, we write $L^{1}_{s}(\RR^N)$ to mean $L^{1}_{s}(\RR^N;\RR)$.
		
		Given $f = (f_1,...,f_M)\in L^1_{\text{loc}}(\RR^N;\RR^M)$, we say that $v\in L^{1}_{s}(\RR^N;\RR^M)$ is a solution to the equation $\mathcal{L}_K v = f$, in the distributional sense if, for every $\varphi \in \mathcal{D}(\RR^N)$, we have: 
		\begin{equation}\label{def:distributional solution}
			\int_{\RR^N} v_i(x)\,\mathcal{L}_K \varphi(x)\,dx = \int_{\RR^N} \varphi(x)\, f_i(x)\,dx \,\, \text{ for all }\,\, i = 1,...,M.
		\end{equation}
		We point out that the assumption that $v\in L^{1}_{s}(\RR^N;\RR^M)$ is needed to ensure that \eqref{def:distributional solution} makes sense. Indeed, leveraging (iii) in Remark \ref{rem:consequences of upper bounds} below, we see that, for any $\varphi\in\mathcal{D}(\RR^N)$, there exists a constant $C$, such that
		\begin{equation*}
			\left|\int_{\RR^N} v_i(x) \mathcal{L}_K\varphi (x)\,dx\right|\leq C\int_{\RR^N}\frac{|v_i(x)|}{1+|x|^{N+2s}}\,dx<\infty.
	\end{equation*}
	
	Now we state a result showing that, when applied to functions in $L^{1}_{s}(\RR^N)$ that are regular enough, yields that the operator~$\mathcal{L}_K$ is well-defined. This is achieved by showing that for functions~$\varphi\in C^2_{\text{loc}}(\RR^N)\cap L^{1}_{s}(\RR^N)$, the function~$\mathcal{L}_K \varphi$ exhibits, at most, polynomial growth. We also show that if $\varphi\in C^2_{c}(\RR^N)$ (i.e., $u$ is of class $C^2$ and has compact support), then $\mathcal{L}_K \varphi$ decays, at infinity, like $|x|^{-N-2s}$. This generalises what is already known for the fractional Laplacian (see, e.g., \cite[Equation $1.9$]{MR3461641}), i.e., when $K(x) = C_{N,s}|x|^{-N-2s}$ (with $C_{N,s}$ being a suitable normalisation constant), and, in fact, the estimates follow by the same line of arguments. 
	\begin{lemma}\label{lem:behaviour of the operator in D}
		Let~$K:\RR^N \rightarrow \RR$ be a measurable, even function satisfying \eqref{eq:comparison of k with FL}.
		For any $\varphi \in L^{1}_{s}(\RR^N)\cap C^{2}_{\text{loc}}(\RR^N)$ and any $r>0$, we have that
		\begin{equation}\label{eq:bound on nonlocal operator I}
			|\mathcal{L}_K \varphi(x)|\leq  C\left(\|D^2 \varphi\|_{L^{\infty}(B_{r}(x))}r^{2-2s}+ |\varphi(x)|r^{-2s}+\|\varphi\|_{L^{1}_{s}(\RR^N)} \left(1+|x|^{N+2s}\right)\right)
		,\end{equation}
		where $C$ is a positive constant depending only on $N$, $s$, $r$ and $\Lambda$.
		
		If, additionally, $\varphi$ has compact support, then there exists $M>0$ (depending only on $\varphi$) such that
		\begin{equation}\label{eq:bound on nonlocal operator II}
			|\mathcal{L}_K \varphi(x)|\leq C \chi_{B_{M}(0)}(x)\left(\|D^2 \varphi\|_{L^{\infty}(\RR^N)}r^2 + \|\varphi\|_{L^{\infty}(\RR^N)}\right)r^{-2s}+  \|\varphi\|_{L^{1}(\RR^N)} \chi_{\RR^N\setminus B_{M}(0)(x)}|x|^{-N-2s},
		\end{equation}
		where $C$ is a positive constant depending only on $N$, $s$, and $\Lambda$.
	\end{lemma}
	\begin{proof}
		See Appendix \ref{Appendix}.
	\end{proof}
	\begin{remark}\label{rem:consequences of upper bounds}
		{\rm \begin{enumerate}[(i)]
				\item A direct consequence of \eqref{eq:bound on nonlocal operator I} is that $\mathcal{L}_K \varphi\in L^{1}_{\text{loc}}(\RR^N)$ for any $\varphi\in L^{1}_{s}(\RR^N)\cap C^{2}_{\text{loc}}(\RR^N)$.
				
				\item As previously mentioned, our proof of Theorem~\ref{thm:boundedness of solutions} follows the outline of \cite[Proposition $2.2$]{MR1659256}.
				One of the key facts used in the proof of \cite[Proposition $2.2$]{MR1659256} is that, if~$v\in \mathcal{D}(\RR^N)$ and, for $R>0$, one defines $v_R(x):=v(R^{-1}x)$, then
				$$
				\Delta v_{R}(x) = R^{-2}\Delta v (R^{-1}x).
				$$ 
				In the model case, where one takes $\mathcal{L}_K$ to be the fractional Laplacian $(-\Delta)^s$, we have something similar, since, in this case, $K$ is a homogeneous function: specifically, it holds that
				$$
				(-\Delta)^s v_{R}(x) = R^{-2s}(-\Delta)^s v(R^{-1}x).
				$$
				In the general case, this needs not hold, since $\mathcal{L}_K$ is not necessarily invariant under scaling.  
				
				However, for the proof of Theorem~\ref{thm:boundedness of solutions}, the estimate in~\eqref{eq:bound on nonlocal operator II} turns out to be sufficient. 
				Indeed, substituting $v_R$ into~\eqref{eq:bound on nonlocal operator II} and setting $r=R$, we obtain:
				\begin{equation}\label{eq:explicit decay}
					\begin{split}
						|\mathcal{L}_K v_R(x)| &\leq C \chi_{B_M(0)}(x)\left(\|D^2 v_R\|_{L^{\infty}(\RR^N)}R^{2}+ \|v_R\|_{L^{\infty}(\RR^N)}\right)R^{-2s}\\
						&+C \|\varphi_R\|_{L^{1}(\RR^N)}\chi_{\RR^N\setminus B_{M}(0)}(x)|x|^{-N-2s}\\
						&\leq C \chi_{B_M(0)}(x)\left(\|D^2 v\|_{L^{\infty}(\RR^N)}+ \|\varphi\|_{L^{\infty}(\RR^N)}\right)R^{-2s}\\
						&+ R^{N}\|\varphi\|_{L^{1}(\RR^N)}\chi_{\RR^N\setminus B_{M}(0)}(x)|x|^{-N-2s}\\
						&\leq C_{1} \left[\chi_{B_M(0)}(x) + \chi_{\RR^N\setminus B_M(0)}(x)\left|\frac{x}{R}\right|^{-N-2s}\right]R^{-2s}\\
						&\leq C_{2} \frac{R^{-2s}}{1+\left|\frac{x}{R}\right|^{N+2s}},
					\end{split}
				\end{equation}
				where $C$ is the constant appearing \eqref{eq:bound on nonlocal operator II}, $C_{1}$ and $C_{2}$ are positive constants depending only on $N$, $s$, $\Lambda$ and $v$.
				
				We point out that \eqref{eq:explicit decay} recovers the well-known estimates in the case $\mathcal{L}_K$ is the fractional Laplacian.
				
				\item Further note that, by taking $\varphi\in \mathcal{D}(\RR^N)$ and setting $R=1$ in \eqref{eq:explicit decay},  
				$$
				\int_{\RR^N} \left|w(x)\mathcal{L}_K\varphi(x)\right|\,dx\leq C_{2}\int_{\RR^N} \frac{|w(x)|}{1+|x|^{N+2s}}\,dx = C_{2} \|w\|_{L^{1}_{s}(\RR^N)},\,\,\text{ for all }\,\,w\in L^{1}_{s}(\RR^N).
				$$
		\end{enumerate}}
	\end{remark}
	
	The second tool we need is an integration by parts formula for the operator~$\mathcal{L}_K$. This will be useful in proving Proposition~\ref{prop: General Nonlocal Kato inequality}. Indeed, the key step in the proof of Proposition~\ref{prop: General Nonlocal Kato inequality} is to integrate by parts the left-hand side of \eqref{eq:general Kato}, use a property of convex functions and integrate by parts again to obtain the right-hand side of \eqref{eq:general Kato}. This tool is presented in the next result.
	
	\begin{lemma}\label{lem:integration by parts for nonlocal operator}
		Let~$K:\RR^N\rightarrow\RR$ be a measurable, even function satisfying \eqref{eq:comparison of k with FL} and let $u\in L^{1}_{s}(\RR^N)\cap C^{2}_{\text{loc}}(\RR^N)$.
		Then, for every $\varphi\in\mathcal{D}(\RR^N)$, we have the following integration by parts formula:
		\begin{equation}\label{eq:integration by parts formula}
			\int_{\RR^N} \varphi(x)\mathcal{L}_K u(x)\,dx =\int_{\RR^N} \mathcal{L}_K\varphi(x) u(x)\,dx.
		\end{equation}
	\end{lemma}
	\begin{proof}
		Let~$u$ and $\varphi$ be as in the statement.\\
		For any $\eps>0$ the following chain holds:
		\begin{equation}\label{eq: step 1 in IBP}
			\begin{split}
				&\int_{\mathbb{R}^N}\int_{\mathbb{R}^N\setminus B_{\eps}(x)} \varphi(x)(u(x)-u(y))K(x-y)\,dy\,dx\\
				&= \int_{\mathbb{R}^N}\int_{\mathbb{R}^N\setminus B_{\eps}(x)} (\varphi(x)-\varphi(y))(u(x)-u(y))K(x-y)\,dy\,dx\\
				&+\int_{\mathbb{R}^N}\int_{\mathbb{R}^N\setminus B_{\eps}(x)}\varphi(y)(u(x)-u(y))K(x-y)\,dy\,dx\\
				&=\int_{\mathbb{R}^N}\int_{\mathbb{R}^N\setminus B_{\eps}(x)} (\varphi(x)-\varphi(y))(u(x)-u(y))K(x-y)\,dy\,dx\\
				&+\int_{\mathbb{R}^N}\int_{\mathbb{R}^N\setminus B_{\eps}(y)}\varphi(y)(u(x)-u(y))K(x-y)\,dx\,dy,
			\end{split}
		\end{equation}
		where the equality
		$$
		\int_{\mathbb{R}^N}\int_{\mathbb{R}^N\setminus B_{\eps}(x)}\varphi(y)(u(x)-u(y))K(x-y)\,dy\,dx = \int_{\mathbb{R}^N}\int_{\mathbb{R}^N\setminus B_{\eps}(y)}\varphi(y)(u(x)-u(y))K(y-x)\,dx\,dy
		$$
		follows by Fubini's Theorem (which we can apply due to~\eqref{eq:step 3 in boundedness}) and the fact that $K$ is even.\\
		Rearranging the terms in \eqref{eq: step 1 in IBP} leads to
		\begin{equation}\label{eq: step 2 in IBP}
			\begin{split}
			&\int_{\mathbb{R}^N}\int_{\mathbb{R}^N\setminus B_{\eps}(x)} \varphi(x)(u(x)-u(y))K(x-y)\,dy\,dx\\
			&= \frac{1}{2}\int_{\mathbb{R}^N}\int_{\mathbb{R}^N\setminus B_{\eps}(x)} (\varphi(x)-\varphi(y))(u(x)-u(y))K(x-y)\,dy\,dx.
			\end{split}
		\end{equation}
		In the same way,  
		\begin{equation}\label{eq:step 3 in IBP}
			\begin{split}
			&\int_{\mathbb{R}^N}\int_{\mathbb{R}^N\setminus B_{\eps}(x)} u(x)(\varphi(x)-\varphi(y))K(x-y)\,dy\,dx\\
			&= \frac{1}{2}\int_{\mathbb{R}^N}\int_{\mathbb{R}^N\setminus B_{\eps}(x)} (\varphi(x)-\varphi(y))(u(x)-u(y))K(x-y)\,dy\,dx.
			\end{split}
		\end{equation}
		Combining \eqref{eq: step 2 in IBP} and \eqref{eq:step 3 in IBP} we have that
		\begin{equation}\label{eq:step 4 in IBP}
			\begin{split}
				\int_{\RR^N}\mathcal{L}_{K}u(x)\varphi(x)\,dx&= \int_{\RR^N}\varphi(x)\left[\lim_{\eps\to 0^+}\int_{\RR^N\setminus B_{\eps}(x)}(u(x)-u(y))K(x-y)\,dy\right]\,dx\\
				&=\lim_{\eps\to 0^+}\int_{\RR^N}\int_{\RR^N\setminus B_{\eps}(x)}\varphi(x)(u(x)-u(y))K(x-y)\,dy\,dx\\
				&=\lim_{\eps\to 0^+}\int_{\RR^N}\int_{\RR^N\setminus B_{\eps}(x)}u(x)(\varphi(x)-\varphi(y))K(x-y)\,dy\,dx\\
				&=\int_{\RR^N}u(x)\left[\lim_{\eps\to 0^+}\int_{\RR^N\setminus B_{\eps}(x)}(\varphi(x)-\varphi(y))K(x-y)\,dy\right]\,dx\\
				&=\int_{\RR^N}u(x)\mathcal{L}_{K}\varphi(x)\,dx,
			\end{split}
		\end{equation}
		where the second equality follows from the Dominated Convergence Theorem which is applicable here due to~\eqref{eq:step 5 in boundedness}. The fourth equality also follows from the Dominated Convergence Theorem, which, in this case, is applicable from \eqref{eq:step 1 in decay} and \eqref{eq:step 2 in decay}.
	\end{proof}
	\begin{remark}\label{rem:partial integration by parts}
		{\rm
			We point out two consequences of Lemma~\ref{lem:integration by parts for nonlocal operator} and its proof that will be useful later on.
			\begin{enumerate}[(i)]
				\item If $u\in L^{1}_{s}(\RR^N)$ and $\varphi\in \mathcal{D}(\RR^N)$ then, capitalizing on \eqref{eq:step 3 in IBP} and on the last three lines of \eqref{eq:step 4 in IBP} we obtain that
				\begin{equation}\label{eq:partial integration by parts}
					\int_{\RR^N}u(x)\mathcal{L}_{K}\varphi(x)\,dx = \frac{1}{2}\lim_{\eps\to 0^+}\int_{\mathbb{R}^N}\int_{\mathbb{R}^N\setminus B_{\eps}(x)} (\varphi(x)-\varphi(y))(u(x)-u(y))K(x-y)\,dy\,dx.
				\end{equation}
				
				\item Given $\psi\in\mathcal{D}(\RR^N)$ and $R>1$, we define $\psi_R(x):=\psi(R^{-1}x)$. Then, combining \eqref{eq:explicit decay} (with $v_{R}=\psi_{R}$) and~\eqref{eq:partial integration by parts} (with $u = \psi_R$ and $\varphi = \psi_{R}$) we have 
				\begin{equation}\label{eq:energy estimates for smooth compact supp functions}
					\begin{split}
						&\iint_{\RR^{2N}}\left|\psi_R(x)-\psi_R(y)\right|^2K(x-y)\,dy\,dx\\
						&=\iint_{\RR^{2N}}\left(\psi_R(x)-\psi_R(y)\right)\left(\psi_R(x)-\psi_R(y)\right)K(x-y)\,dy\,dx\\
						&=2\int_{\RR^N}\psi_R(x)\mathcal{L}_K \psi_{R}(x)\,dx\leq 2C_{2} R^{-2s}\int_{\RR^N}\frac{|\psi_R(x)|}{1+\left|\frac{x}{R}\right|^{N+2s}}\,dx\\
						&=2C_{2}R^{N-2s}\|\psi\|_{L^{1}_{s}(\RR^N)},  
					\end{split}
				\end{equation} 
				where $C_{2}$ is the constant that appears in \eqref{eq:explicit decay}, and depends on $N$, $s$, $\Lambda$, and $\psi$.
			\end{enumerate}
			
		}
	\end{remark}
	As mentioned in the Introduction, the proof of Proposition~\ref{prop: General Nonlocal Kato inequality} will be achieved through a density argument. More precisely, given $u\in L^{1}_{s}(\RR^N)$ we will approximate $u$ by a sequence $\{u_l\}_{l\in\NN} \subset\mathcal{D}(\RR^N)$, prove Proposition~\ref{prop: General Nonlocal Kato inequality} and then take the limit. However, to perform the density argument, one cannot just use the density of $\mathcal{D}(\RR^N)$ in $L^{1}_{s}(\RR^N)$ (which is immediate given that $L^{1}_{s}(\RR^N)$ is a weighted $L^1$ space). This is because approximating a function $u\in L^{1}_{s}(\RR^N)$ with a sequence $\{u_l\}_{l\in \mathbb{N}}\subset \mathcal{D}(\RR^N)$ does not guarantee that $\mathcal{L}_K u_l \rightarrow \mathcal{L}_K u$ in $L^{1}_{\text{loc}}(\RR^N)$. As a consequence, to perform an approximation argument, we employ mollifiers; this will allow us to approximate both the function $u$ in $ L^{1}_{s}(\RR^N)$ and $\mathcal{L}_K u$ in $L^{1}_{\text{loc}}(\RR^N)$. 
	As such, we need to study the behaviour of the operator~$\mathcal{L}_K$ with respect to the convolution between two functions.
	\begin{lemma}\label{lem:convolution and the operator}
		Let~$K:\RR^N\rightarrow\RR$ be a measurable, even function satisfying \eqref{eq:comparison of k with FL}. For any $u\in L^{1}_{s}(\RR^N)\cap C^{2}_{\text{loc}}(\RR^N)$ and any $\varphi\in \mathcal{D}(\RR^N)$, we have
		\begin{equation*}
			\mathcal{L}_K \left(u\ast \varphi\right)= \left(\mathcal{L}_K u\right)\ast \varphi  = u\ast\left(\mathcal{L}_K \varphi\right).
		\end{equation*}
	\end{lemma}
	\begin{proof}
		For $u$ and $\varphi$ as in the statement and $x\in\RR^N$, we have that
		\begin{equation*}
			\begin{split}
				\mathcal{L}_K\left(u\ast \varphi\right)(x)&= \int_{\RR^N}\left[u\ast\varphi(x)-u\ast\varphi(y)\right]\,K(x-y)\,dy\\
				&=\int_{\RR^N}\left[\int_{\RR^N}\left(u(x-z)-u(y-z)\right)\varphi(z)\,dz\right]\,K(x-y)dy\\
				&=\int_{\RR^N}\int_{\RR^N}\left(u(x-z)-u(y-z)\right)\varphi(z)\,K(x-y)\,dz\,dy\\
				&=\int_{\RR^N}\varphi(z)\int_{\RR^N}\left(u(x-z)-u(y-z)\right)\,K(x-y)\,dy\,dz\\
				&=\int_{\RR^N}\varphi(z)\int_{\RR^N}\left(u(x-z)-u(w)\right)\,K(x-w-z)\,dw\,dz\\
				&=\int_{\RR^N}\varphi(z)\mathcal{L}_K u(x-z)\,dz= \varphi\ast \left(\mathcal{L}_K u\right)(x),
			\end{split}
		\end{equation*}
		where, in the fourth line, we used Fubini's Theorem, and in the fifth line, we performed the change of variable $w=y-z$.
		Analogously, we find
		\begin{equation*}
		\mathcal{L}_K \left(u\ast \varphi\right)= u\ast\left(\mathcal{L}_K \varphi\right).\qedhere
		\end{equation*}
	\end{proof}

	\section{Proof of Proposition~\ref{prop: General Nonlocal Kato inequality}}\label{Section 3}
	This section aims to provide a proof of Proposition~\ref{prop: General Nonlocal Kato inequality}.
	But first, we show that the composition with $C^{2}_{\text{loc}}(\RR)$ convex Lipschitz functions preserves the space $L^{1}_{s}(\RR^N)$. We also prove a pointwise counterpart to~\eqref{eq:general Kato} for smooth functions.
	
	\begin{lemma}\label{lem:easy gen Kato} Let~$K:\RR^N \rightarrow \RR$ be a measurable, even function satisfying \eqref{eq:comparison of k with FL}, and let $\Phi\in C^{2}_{\text{loc}}(\RR)$ be a convex Lipschitz function. 
		Then, for any $u\in L^{1}_{s}(\RR^N)\cap C^{2}_{\text{loc}}(\RR^N)$, we have that $\Phi(u)\in L^{1}_{s}(\RR^N)\cap C^{2}_{\text{loc}}(\RR^N)$ and 
		\begin{equation}\label{eq:estimate on norm of convex comp}
			\|\Phi(u)\|_{L^{1}_{s}(\RR^N)}\leq \|\Phi'\|_{L^{\infty}(\RR)}\|u\|_{L^{1}_{s}(\RR)}+ C|\Phi(0)|,
		\end{equation}
		where $C$ is a positive constant depending only on $N$ and $s$.
		
		Furthermore, the following inequality holds: 
		\begin{equation}\label{eq:pointwise Kato}
			\mathcal{L}_K \Phi(u)(x)\leq \Phi'(u(x))\mathcal{L}_Ku(x),\,\,\text{ for all }\,\,x\in\RR^N.
		\end{equation}
	\end{lemma}
	\begin{proof}
		Given $u\in L^{1}_{s}(\RR^N)\cap C^{2}_{\text{loc}}(\RR^N)$, we have that~$\Phi(u)\in C^{2}_{\text{loc}}(\RR^N)$ since $\Phi\in C^{2}_{\text{loc}}(\RR)$. 
		Furthermore, leveraging the assumption that $\Phi'\in L^{\infty}(\RR)$, we see that
		\begin{equation*}
			\int_{\RR^N}\frac{|\Phi(u(x))|}{1+|x|^{N+2s}}\,dx\leq \int_{\RR^N}\frac{|\Phi(u(x))-\Phi(0)| +|\Phi(0)|}{1+|x|^{N+2s}}\,dx\leq \int_{\RR^N}\frac{\|\Phi'\|_{L^{\infty}(\RR)}|u(x)|+|\Phi(0)|}{1+|x|^{N+2s}}\,dx,
		\end{equation*}
		from which \eqref{eq:estimate on norm of convex comp} follows by taking $C=\displaystyle\int_{\RR^N}\frac{1}{1+|x|^{N+2s}}\,dx<\infty$.
		
		Then, given that $\Phi$ is convex, we have, for any $t_1,t_2\in\RR$,
		\begin{equation*}
			\Phi(t_1) -\Phi(t_2)\leq \Phi'(t_1)(t_1-t_2).
		\end{equation*}
		Applying this inequality with $t_1=u(x)$ and $t_2 = u(y)$ we see that 
		\begin{equation*}
			\Phi(u(x)) -\Phi(u(y))\leq \Phi'(u(x))(u(x)-u(y)),
		\end{equation*}
		and therefore, 
		\begin{equation*}
			\begin{split}
				\mathcal{L}_K\Phi(u)(x)&=\int_{\RR^N}\left(\Phi(u(x))-\Phi(u(y))\right)K(x-y)\,dy\\
				&\leq \Phi'(u(x))\int_{\RR^N}\left(u(x)-u(y)\right)K(x-y)\,dy = \Phi'(u(x))\mathcal{L}_K u(x).\qedhere
			\end{split}
		\end{equation*}
	\end{proof}
	Now we are ready to prove the general nonlocal Kato-type inequality in Proposition~\ref{prop: General Nonlocal Kato inequality}.
	\begin{proof}[Proof of Proposition~\ref{prop: General Nonlocal Kato inequality}]
		Let~$u\in L^1_{s}(\RR^N)$ be such that $\mathcal{L}_K u = f \in L^1_{\text{loc}}(\RR^N)$, in the sense of distributions, i.e.,
		\begin{equation}\label{eq:distributional solution}
			\int_{\RR^N} f(x)\varphi(x)\,dx = \int_{\RR^N} u(x) \mathcal{L}_K \varphi(x)\,dx,\,\,\text{ for every }\,\,\varphi\in \mathcal{D}(\RR^N). 
		\end{equation}
		Let~$\eta\in \mathcal{D}(\RR^N)$ be such that 
		\begin{equation}
			0\leq \eta\leq 1,\,\, \int_{\RR^N}\eta(x)\,dx =1\,\,\text{and}\,\, \text{supp}(\varphi)\subset B_1(0),
		\end{equation}
		and, for all $m\in \NN$, define $\eta_{m}(x):=m^{N}\eta\left(mx\right)$. By the standard properties of convolutions with mollifiers and Lemma~\ref{lem:density in L1s} we have (up to a subsequence) that $u_m:=u\ast\eta_{m} \in C^{\infty}(\RR^N)\cap L^{1}_{s}(\RR^N)$ and $\|u_{m}-u\|_{L^{1}_s(\RR^N)}\to 0$ as $m\to \infty$.
		
		Now let $\Phi:\RR \rightarrow \RR$ be a Lipschitz, convex function. In light of Lemma~\ref{lem:approx of convex Lipschitz functions} (with $N=1$), we find a sequence $\Phi_{n}:\RR\rightarrow \RR$ ($n\in\mathbb{N}$) of smooth, convex functions, satisfying:
		\begin{enumerate}
			\item
			$
			\|\Phi_n'\|_{L^{\infty}(\RR)}\leq \|\Phi'\|_{L^{\infty}(\RR)}; 
			$
			\item 
			$
			\lim_{n\to \infty}\Phi_n'(t)\to \Phi'(t),\,\,\text{ for a.e. }\,\, t\in \RR;
			$
			\item 
			$
			\lim_{n\to \infty}\|\Phi_n-\Phi\|_{L^{\infty}(\RR)}= 0.
			$
		\end{enumerate}
		Given that $\|\Phi_{n}'\|_{L^{\infty}(\RR)}$ is uniformly bounded in $n$, we see that, for every $n\in\NN$ (and up to a subsequence in $m$), 
		$$\lim_{m\to \infty}\left\|\Phi_n(u_{m})-\Phi_n(u)\right\|_{L^{1}_{s}(\RR^N)}=0.
		$$
		For given $n\in\NN$, and $\varphi\in\mathcal{D}(\RR^N)$ we thus have:
		\begin{equation}\label{eq:gen Kato step 0}
			\begin{split}
				\int_{\RR^N} \Phi_{n}(u(x))\mathcal{L}_K \varphi(x)\,dx&=\lim_{m\to \infty}\int_{\RR^N}\Phi_{n}(u_{m}(x))\mathcal{L}_K \varphi(x)\,dx\\
				&=\lim_{m\to \infty}\int_{\RR^N}\mathcal{L}_K(\Phi_{n}(u_{m}))(x)\varphi(x)\,dx\\
				&\leq \lim_{m\to \infty}\int_{\RR^N}\Phi_{n}'(u_{m}(x))\mathcal{L}_K u_{m}(x)\varphi(x)\,dx\\
				&=\lim_{m\to \infty}\int_{\RR^N} u_{m}(x)\mathcal{L}_K\left(\Phi_{n}'(u_{m}) \varphi\right)(x)\,dx.
			\end{split}
		\end{equation}
		Here above, the first equality follows from \eqref{eq:explicit decay} (with $\varphi = v$ and $R=1$), the fact that $\Phi_n(u_m)\to \Phi_n(u)$, in $L^{1}_{s}(\RR^N)$, as $m\to \infty$, and the Dominated Convergence Theorem. 
		In the second and fourth lines, we employed Lemma~\ref{lem:integration by parts for nonlocal operator}. In the third line, we used \eqref{eq:pointwise Kato} in Lemma~\ref{lem:easy gen Kato}.\\
		Hence, denoting by $\rho_m$ the function given by $\rho_m(x)= m^{N}\eta\left(-m x\right)$, we follow up on \eqref{eq:gen Kato step 0} to obtain that
		\begin{equation}\label{eq:gen kato step 1}
			\begin{split}
				\int_{\RR^N} \Phi_{n}(u(x))\mathcal{L}_K \varphi(x)\,dx&\leq\lim_{m\to \infty}\int_{\RR^N} u\ast\eta_{m}(x)\mathcal{L}_K\left(\Phi_{n}'(u_{m}) \varphi\right)(x)\,dx\\
				&=\lim_{m\to \infty}\int_{\RR^N} u(z)\left(\rho_m\ast\mathcal{L}_K\left(\Phi_{n}'(u_{m})\varphi\right)\right)(z)\,dz\\
				&= \lim_{m\to \infty}\int_{\RR^N} u(z)\mathcal{L}_K\left(\rho_{m}\ast\left(\Phi_{n}'(u_{m})\varphi\right)\right)(z)\,dz\\
				&= \lim_{m\to \infty}\int_{\RR^N} f(z)\rho_{m}\ast\left(\Phi_{n}'(u_{m})\varphi\right)(z)\,dz\\
				&=\lim_{m\to \infty}\int_{\RR^N} \left(\eta_{m}\ast f\right)(x)\Phi_{n}'(u_{m}(x))\varphi(x)\,dx,
			\end{split}
		\end{equation}
		where, in the third line, we used Lemma~\ref{lem:convolution and the operator} and in the fourth line, we leveraged \eqref{eq:distributional solution}.\\
		Recalling that $\varphi\in\mathcal{D}(\RR^N)$, $\eta_{m}\ast f \to f$ in $L^{1}_{\text{loc}}(\RR^N)$ (as $m\to \infty$), $\Phi_{n}'\in L^{\infty}(\RR)$ and $\Phi_{n}'(u_m(z))\to \Phi_{n}'(u(z))$ (as $m\to \infty$) for a.e. $z\in\RR^N$, we obtain
		\begin{equation}\label{eq:gen kato step 2}
			\begin{split}
				\lim_{m\to \infty}\int_{\RR^N} \left(\eta_{m}\ast f\right)(x)\Phi_{n}'(u_{m}(x))\varphi(x)\,dx &= \lim_{m\to \infty}\int_{\RR^N}\left(\eta_{m}\ast f(x)-f(x)\right)\Phi_{n}'(u_{m}(x))\varphi(x)\,dx\\
				&+\lim_{m\to \infty}\int_{\RR^N}f(x)\left(\Phi_{n}'(u_{m}(x))-\Phi_{n}'(u(x))\right)\varphi(x)\,dx\\
				&+\int_{\RR^N}f(x)\Phi_{n}'(u(x))\varphi(x)\,dx\\
				&= \int_{\RR^N}f(x)\Phi_{n}'(u(x))\varphi(x)\,dx.
			\end{split}
		\end{equation}
		
		Hence, combining \eqref{eq:gen kato step 1} and \eqref{eq:gen kato step 2} we get, for any $n\in\NN$,
		\begin{equation*}
			\int_{\RR^N} \Phi_{n}(u(x))\mathcal{L}_K \varphi(x)\,dx\leq \int_{\RR^N}f(x)\Phi_{n}'(u(x))\varphi(x)\,dx.
		\end{equation*}
		
		Since $\Phi_{n}$ converges uniformly to $\Phi$, $\|\Phi_{n}'\|_{L^{\infty}(\RR)}$ is uniformly bounded and $\Phi_{n}'(u(x))\to \Phi'(u(x))$ (as $n\to \infty$) for a.e. $x\in \RR^N$, using the Dominated Convergence Theorem we conclude that 
		\begin{equation*}
			\begin{split}
				\int_{\RR^N}\Phi(u(x))\mathcal{L}_K \varphi(x)\,dx&=\lim_{n\to \infty}\int_{\RR^N}\Phi_n(u(x))\mathcal{L}_K \varphi(x)\,dx\leq \lim_{n\to \infty}\int_{\RR^N} \Phi_{n}'(u(x))f(x)\varphi(x)\,dx\\
				&= \int_{\RR^N}\Phi'(u(x))f(x)\varphi(x)\,dx.\qedhere
			\end{split}
		\end{equation*} 
	\end{proof}
	\section{Proof of Theorem~\ref{thm:boundedness of solutions}}\label{Section 4}
	This section is devoted to the proof of Theorem~\ref{thm:boundedness of solutions}.
	First, we present a result that extends~\cite[Lemma~$2.1$]{MR1659256} to solutions of \eqref{Main Problem}. Apart from a few technical changes, the proof follows the same line of reasoning as the proof of \cite[Lemma~$2.1$]{MR1659256}, but we include it in the Appendix~\ref{Appendix} for the readers' convenience. 
	
	\begin{lemma}\label{lem:distributional ineq}
		Let~$K:\RR^N\rightarrow \RR$ be a measurable, even function satisfying \eqref{eq:comparison of k with FL} and	let $u\in L^{1}_{s}(\RR^N;\RR^M)$ be a distributional solution of \eqref{Main Problem}. Then, there exists $\alpha>0$ such that, for any $i = 1,...,M$ and any nonnegative $\varphi\in \mathcal{D}(\RR^N)$, we have that
		\begin{equation}\label{eq:distributional ineq 2}
			\int_{\RR^N} \left(|u_i(x)|-\sqrt{k_n}\right)^+ \mathcal{L}_K\varphi(x)\,dx \leq-\alpha \int_{\RR^N}\left(|u_i|-\sqrt{k_n}\right)^+\varphi(x)\,dx.
		\end{equation}
	\end{lemma}
	\begin{proof}See Appendix \ref{Appendix}.
	\end{proof}
	
	Now we are ready to prove Theorem~\ref{thm:boundedness of solutions}. Our proof can now follow the same outline as the proof of \cite[Proposition $2.2$]{MR1659256}, by relying now on the auxiliary results that we have developed in the previous analysis.
	
	\begin{proof}[Proof of Theorem~\ref{thm:boundedness of solutions}]
		Given $i\in \{1,...,M\}$, we set $\phi_{i,n}:=\left(|u_i|-\sqrt{k_n}\right)^+$. 
		
		First, we prove that if either \eqref{assumption 1 on sol} or \eqref{assumption 2 on sol} hold, then there exist a point $x_0\in \RR^N$, positive constants $C$, $R_0$ and $\lambda<N+2s$ (all independent of~$R$) such that, for all $R>R_0$:
		\begin{equation}\label{eq:step 1 in bounded sol}
			\int_{B_{R}(x_0)}\phi_{i,n}(x)\,dx\leq CR^{\lambda}.
		\end{equation}
		
		If \eqref{assumption 2 on sol} holds, then \eqref{eq:step 1 in bounded sol} is immediate. 
		
		If, on the other hand, \eqref{assumption 1 on sol} holds, then we consider $\xi \in \mathcal{D}(\RR^N)$ such that 
		$$
		\text{supp}(\xi)\subset B_{2}(0), \,\,0\leq \xi\leq 1\,\,\text{and}\,\,\xi_{\vert_{B_{1}(0)}} \equiv 1,
		$$	
		and we define $\xi_R(x):= \xi\left(R^{-1}(x-x_0)\right)$.\\
		Then, using Lemma~\ref{lem:distributional ineq} and \eqref{eq:partial integration by parts} we see that
		\begin{equation}\label{eq:step 1.1 in bounded sol}
			\begin{split}&\!\!\!\!\!\!
				\int_{B_{R}(x_0)}\phi_{i,n}(x)\,dx\\&\leq \int_{\RR^N}\phi_{i,n}(x)\xi_R(x)\,dx\\&\leq -\alpha^{-1}\int_{\RR^N}\phi_{i,n}(x)\mathcal{L}_K \xi_R(x)\,dx\\
				&=-\frac{1}{2\alpha}\lim_{\eps\to 0^+}\int_{\RR^N}\int_{\RR^N\setminus B_{\eps}(x)}\left(\phi_{i,n}(x)-\phi_{i,n}(y)\right)\left(\xi_R(x)-\xi_R(y)\right)\,K(x-y)\,dy\,dx\\
				&\leq \frac{1}{2\alpha}\lim_{\eps\to 0^+}\int_{\RR^N}\int_{\RR^N\setminus B_{\eps}(x)}\left|\phi_{i,n}(x)-\phi_{i,n}(y)\right|\left|\xi_R(x)-\xi_R(y)\right|\,K(x-y)\,dy\,dx\\
				&\leq \frac{1}{2\alpha}\int_{\RR^N}\int_{\RR^N}\left|\phi_{i,n}(x)-\phi_{i,n}(y)\right|\left|\xi_R(x)-\xi_R(y)\right|\,K(x-y)\,dy\,dx.
			\end{split}
		\end{equation}
		Noting that, for $(x,y)\in (\RR^N\times \RR^N)\setminus T_{B_{2R}(x_0)}$ (where $T_{B_{2R}(x_0)}$ is as defined in the statement of Theorem~\ref{thm:boundedness of solutions}), it holds that~$\xi_{R}(x)-\xi_{R}(y)=0$, we use Hölder's inequality, \eqref{assumption 1 on sol} with~$R$ replaced by~$2R$ and \eqref{eq:energy estimates for smooth compact supp functions} to obtain that
		\begin{equation}
			\begin{split}
				&\int_{B_{R}(x_0)}\phi_{i,n}(x)\,dx\leq\frac{1}{2\alpha}\iint_{T_{B_{2R}(x_0)}}\left|\phi_{i,n}(x)-\phi_{i,n}(y)\right|\left|\xi_R(x)-\xi_R(y)\right|\,K(x-y)\,dy\,dx\\
				&\leq \frac{1}{2\alpha}\left(\iint_{T_{B_{2R}(x_0)}}\left|\phi_{i,n}(x)-\phi_{i,n}(y)\right|^2\,K(x-y)\,dy\,dx\right)^{1/2}\\
				&\times\left(\iint_{T_{B_{2R}(x_0)}}\left|\xi_R(x)-\xi_R(y)\right|^{2}\,K(x-y)\,dy\,dx\right)^{1/2}\\
				&\leq \frac{\left(2C C_{2}\right)^{1/2}}{2\alpha}\|\xi\|_{L^{1}_{s}(\RR^N)}R^{\frac{h+N-2s}{2}},
			\end{split}
		\end{equation}
		where $C_{2}$ is the constant that appears in \eqref{eq:energy estimates for smooth compact supp functions} and $C$ is the constant that appears in \eqref{assumption 1 on sol}.
		Taking into account that (in \eqref{assumption 1 on sol}) $h<N+6s$, the claim in~\eqref{eq:step 1 in bounded sol} follows with $\lambda = \frac{h+N-2s}{2}$.
		
		Now we use mathematical induction to prove that, for each nonnegative integer $m$, there exists a positive constant $A_m$ such that, for all $R>\max\{R_0, |x_0|\}$,
		\begin{equation}\label{eq:step 2 in bounded sol}
			\int_{B_{R}(x_0)}\phi_{i,n}(x)\,dx \leq A_m R^{\lambda-2sm}.
		\end{equation}
		Letting $\xi_R$ be as before and leveraging Lemma~\ref{lem:distributional ineq} and \eqref{eq:explicit decay} we obtain 
		\begin{equation}\label{eq:step 2.1 in bounded sol}
			\begin{split}
				&\int_{B_{R}(x_0)}\phi_{i,n}(x)\,dx\leq \int_{\RR^N} \phi_{i,n}\xi_{R}(x)\,dx \leq -\alpha^{-1}\int_{\RR^N} \phi_{i,n}(x) \mathcal{L}_K\xi_R(x)\,dx\\
				&\leq \alpha^{-1}R^{-2s}C_{2}\int_{\RR^N}\frac{\phi_{i,n}(x)}{1+\left|\frac{x}{R}\right|^{N+2s}}\,dx\\
				&\leq \alpha^{-1}C_{2}R^{-2s}\left(\int_{B_{R}(x_0)}\phi_{i,n}(x)\,dx+\sum_{k=1}^{\infty}\int_{B_{2^{k}R}(x_0)\setminus B_{2^{k-1}R}(x_0)}\frac{\phi_{i,n}(x)}{1+\left|\frac{x}{R}\right|^{N+2s}}\,dx\right),
			\end{split}
		\end{equation}
		where $C_{2}$ is the constant that appears in \eqref{eq:explicit decay}.
		
		Since $R>|x_0|$, simple geometric considerations imply, for any $k>1$: 
		\begin{equation*}
			\max_{x\in B_{2^{k}R}(x_0)\setminus B_{2^{k-1}R}(x_0)}\frac{1}{1+\left|\frac{x}{R}\right|^{N+2s}} = \frac{1}{1+\left(2^{k-1}-\frac{|x_0|}{R}\right)^{N+2s}}\leq C_{3} 2^{-(k-1)(N+2s)},
		\end{equation*} 
		for some positive constant $C_{3}$ depending only on $N$ and $s$.
		
		This, in combination with \eqref{eq:step 2.1 in bounded sol} and the induction hypothesis, leads to 
		\begin{equation}\label{eq:step 2.2 in bounded sol}
			\begin{split}
				&\int_{B_{R}(x_0)}\phi_{i,n}(x)\,dx\\
				&\leq \alpha^{-1}C_{2}R^{-2s}\left(A_{m-1}R^{\lambda -2s(m-1)}+\sum_{k=1}^{\infty}\int_{B_{2^{k}R}(x_0)\setminus B_{2^{k-1}R}(x_0)}2^{-(k-1)(N+2s)}\phi_{i,n}(x)\,dx\right)\\
				&\leq \alpha^{-1}C_{2}C_{3} R^{-2s}A_{m-1}\left(R^{\lambda-2s(m-1)}+\sum_{k=1}^{\infty}2^{-(k-1)(N+2s)}\left(2^{k}R\right)^{\lambda-2s(m-1)}\right)\\
				&\leq \alpha^{-1}C_{2}C_{3}A_{m-1} R^{\lambda-2sm}\left(1+\sum_{k=1}^{\infty} 2^{-(k-1)(N+2s)+k(\lambda-2s(m-1))}\right)
			\end{split}
		\end{equation}
		
		Since, $m\geq 1$ and $\lambda<N+2s$, by taking 
		$$
		A_{m}= \alpha^{-1}2^{N+2s}C_{2}C_{3}A_{m-1} \left(1+\sum_{k=1}^{\infty} 2^{-(k-1)(N+2s)+k(\lambda-2s(m-1))}\right),
		$$
		we obtain \eqref{eq:step 2 in bounded sol} with $m$.
		
		In particular, taking $m>\frac{\lambda}{2s}$ and taking the limit $R\to \infty$ we conclude that $\phi_{i,n}\equiv 0$, i.e., $|u_i|\leq \sqrt{k_n}$, from which \eqref{eq:bound on solution} follows immediately.
	\end{proof}
	\begin{remark}\label{rem:necessity for aditional assumptions}
		{\rm In this remark, we show that \eqref{eq:integrability assumption on sol} is not enough to derive the conclusion of Theorem~\ref{thm:boundedness of solutions} with the arguments we employed. 
	 	  Note that for any $x_0\in \RR^N$ and $R>|x_0|$, assumption \eqref{eq:integrability assumption on sol} yields that 
			\begin{equation}\label{eq:not an alternative}
				\begin{split}
					&\int_{B_{R}(x_0)} \phi_{i,n}(x)\,dx\\
					&\leq \int_{B_{R+|x_0|}(0)}\phi_{i,n}(x)\,dx\leq 2(R+|x_0|)^{N+2s}\int_{B_{R}(0)}\frac{\phi_{i,n}(x)}{1+|x|^{N+2s}}\,dx\\
					&\leq  2^{N+1+2s}R^{N+2s}\int_{\RR^N}\frac{\phi_{i,n}(x)}{1+|x|^{N+2s}}\,dx \leq 2^{N+1+2s}R^{N+2s}\int_{\RR^N}\frac{|u(x)|+\sqrt{k_n}}{1+|x|^{N+2s}}\,dx\\
					&\leq 2^{N+1+2s}\left[\|u\|_{L^{1}_{s}(\RR^N;\RR^M)}+C_{N,s,k_n}\right]R^{N+2s},
				\end{split}
			\end{equation}
			where $C_{N,s,k_n}$ is given by 
			$$
			C_{N,s,k_n} = \int_{\RR^N}\frac{\sqrt{k_n}}{1+|x|^{N+2s}}\,dx.
			$$
			If one attempts to execute the induction argument in the proof of Theorem~\ref{thm:boundedness of solutions} with \eqref{eq:not an alternative} instead of \eqref{eq:step 1 in bounded sol}, the argument breaks down. Indeed, repeating \eqref{eq:step 2.1 in bounded sol} and \eqref{eq:step 2.2 in bounded sol},  with $m=1$ and $\lambda = N+2s$ we obtain
			\begin{equation*}
				\begin{split}
					\int_{B_{R}(0)} \phi_{i,n}(x)\,dx&\leq C R^{N}\left(1+\sum_{k=1}^{\infty} 2^{-(k-1)(N+2s)+(N+2s)k}\right)\\
					&\leq C R^{N}\left(1+\sum_{k=1}^{\infty} 2^{N+2s}\right) =\infty,
				\end{split}
			\end{equation*}
			where, in the first inequality $C$ is a constant depending only on $N$, $s$, $\|u\|_{L^{1}_{s}(\RR^N;\RR^M)}$, $k_n$ and $\alpha$. This last estimate breaks the induction argument.
		}
	\end{remark}

 \section{Proof of Corollary \ref{cor:fractional DeGiorgi in 2D}}\label{sec:proof of De Giorgi}
 
 \begin{proof}
 	 By Theorem~\ref{thm:boundedness of solutions} we have that $u\in L^{\infty}(\RR^2)$ with 
		$$
		\left\|u\right\|_{L^{\infty}(\RR^2)}\leq \sqrt{k_n}.
		$$ 
		Therefore, we have by a bootstrap application of fractional elliptic regularity theory (see e.g. \cite{MR4769823}) that $u$ has bounded derivative and therefore 
		 \begin{equation*}
		 	\int_{\RR^2}\frac{\left|\partial_{x_2}u(x-ty)\right|}{1+|y|^{2+2s}}\,dy<\infty,\quad \text{for all}\,\, x\in\RR^2 \,\,\text{and}\,\, t>0.
		 \end{equation*}
		
		In particular, $u$ satisfies the assumptions of \cite[Theorem $1.2$]{MR4968697}. This implies the desired result. 
 \end{proof}
	\appendix
	\counterwithin*{equation}{section}
	\renewcommand\theequation{\thesection.\arabic{equation}}
	\section{Auxiliary and technical results and proof of some simple lemmata}\label{Appendix}
	
	The first result we present refers to the approximation of functions in $L^{1}_{s}(\RR^N)$ by means of convolution with mollifiers, for a proof, see e.g.,\cite[Proposition A.$1$]{MR4942118}. 
	\begin{lemma}\label{lem:density in L1s}
		Let~$u\in L^{1}_{s}(\RR^N)$ and let $\eta\in\mathcal{D}(\RR^N)$ be such that 
		\begin{equation}
			0\leq \eta\leq 1,\,\, \int_{\RR^N}\eta(x)\,dx =1\,\,\text{and}\,\, \text{supp}(\varphi)\subset B_1(0),
		\end{equation}
		and, for $n\in \NN$, define $\eta_{n}(x):=n^{N}\eta\left(nx\right)$.\\
		Then, for every $n\in \NN$ we have $u\ast \eta_{n}\in C^{\infty}(\RR^N)\cap L^{1}_{s}(\RR^N)$ and, up to a subsequence, 
		\begin{equation*}
			\|u\ast\eta_{n}-u\|_{L^{1}_{s}(\RR^N)}\to 0,\,\,\text{ as }\,\, n\to\infty.
		\end{equation*}
	\end{lemma}
	
	The second result we present states that we can approximate, in a suitable sense, a convex Lipschitz function by a sequence of smooth, convex functions with uniformly bounded derivatives.
	\begin{lemma}\label{lem:approx of convex Lipschitz functions}
		Let~$F:\RR^N\rightarrow \RR$ be a convex Lipschitz function.
		Let~$\eta\in \mathcal{D}(\RR^N)$ be a nonnegative function satisfying $\int_{\RR^N} \eta(x)\,dx = 1$.
		For $n\in\NN $, we define the functions $\eta_{n}(x):=n^{N}\eta(nx)$.
		Then, $F_n := F\ast\eta_{n}$ is a smooth, convex function satisfying
	 \begin{equation}\label{eq:property 1 of approx convex}
				\|\nabla F_n\|_{L^{\infty}(\RR^N)}\leq \|\nabla F\|_{L^{\infty}(\RR^N)}.
			\end{equation}
			Also, up to a subsequence, 
			\begin{equation}\label{eq:property 2 of approx convex}
				\lim_{n\to \infty }\nabla F_{n}(x)= \nabla F(x),\,\,\text{ for a.e. }\,\,x\in\RR^N.
			\end{equation}
			Moreover
			\begin{equation}\label{eq:property 3 of approx convex}
				\lim_{n\to\infty}\|F_{n}-F\|_{L^{\infty}(\RR^N)}\to 0.
			\end{equation}
	\end{lemma}
	\begin{proof}
		The smoothness of $F_n$ is a direct consequence of the smoothness of $\eta$ and the properties of the convolution of two functions. The convexity of $F_n$ follows from the convexity of $F$ and the nonnegativity of $\eta$. Indeed, for any $x,,z\in\RR^N$ and any $t\in[0,1]$,
		\begin{equation*}
			\begin{split}
				F_n(tx+(1-t)z)&=\int_{\RR^N}\eta_{n}(y)F(tx+(1-t)z-y)\,dy\\
				&\leq  \int_{\RR^N}\eta_{n}(y)\left(tF(x-y)+(1-t)F(z-y)\right)\,dy\\
				&=t F_{n}(x) +(1-t)F_{n}(z).
			\end{split}
		\end{equation*}
		
		The standard properties of the convolution between two functions lead to~\eqref{eq:property 1 of approx convex}, since, for a.e. $x\in\RR^N$,
		\begin{equation*}
			\left|\nabla F_{n}(x)\right|=\left|\eta_{n}\ast \nabla F(x)\right|\leq \int_{\RR^N}\eta_{n}(y)\left|\nabla F(x-y)\right|\,dy\leq \|\nabla F\|_{L^{\infty}(\RR^N)}.
		\end{equation*}
		Since $L^{\infty}(\RR^N)\subset L^{1}_{\text{loc}}(\RR^N)$, the claim \eqref{eq:property 3 of approx convex} follows immediately from the well-known fact that $\nabla F_n = \nabla F \ast \eta_{n} \rightarrow \nabla F$ (as $n\to \infty$) in $L^{1}_{\text{loc}}(\RR^N)$. 
		
		Using the fact that $\eta$ is nonnegative and satisfies $\int_{\RR^N}\eta(y) \,dy=1$, we see that
		\begin{equation*}
			\begin{split}
				\left|F\ast \eta_{n}(x)-F(x)\right|&= \left|\int_{\RR^N}\eta_{n}(y)\left[F(x-y)-F(x)\right]\,dy\right|\leq \|\nabla F\|_{L^{\infty}(\RR^N)}\int_{\RR^N}\eta_{n}(y)|y|\,dy\\
				&=\frac{1}{n}\|\nabla F\|_{L^{\infty}(\RR^N)}\int_{\RR^N}\eta(z)|z|\,dz.
			\end{split}
		\end{equation*}
		Given that $\eta\in\mathcal{D}(\RR^N)$, then $C_{4} = \int_{\RR^N}\eta(z)|z|\,dz <\infty$, and thus the arbitrariness of $x\in \RR^N$ implies \eqref{eq:property 2 of approx convex}, since
		\begin{equation*}
			\lim_{n\to \infty}\left\|F\ast \eta_{n}-F\right\|_{L^{\infty}(\RR^N)} \leq \lim_{n\to \infty} C_{4}\|\nabla F\|_{L^{\infty}(\RR^N)}\frac{1}{n} =0.\qedhere
		\end{equation*} 
	\end{proof}
	
	\begin{proof}[Proof of Lemma~\ref{lem:behaviour of the operator in D}]
		Let~$\varphi\in L^{1}_{s}(\RR^N)\cap C^{2}_{\text{loc}}(\RR^N)$. For any $x\in \RR^N$ and any $r>\eps>0$ we have that
		\begin{equation}\label{eq:step 1 in boundedness}
			\begin{split}
				\int_{B_{r}(x)\setminus B_{\eps}(x)}(\varphi(x)-\varphi(y))\,K(x-y)\,dy&= \int_{B_{r}(0)\setminus B_{\eps}(0)}(\varphi(x)-\varphi(x+z)))\,K(z)\,dz\\
				&= \frac{1}{2}\int_{B_{r}(0)\setminus B_{\eps}(0)}(\varphi(x)-\varphi(x+z))\,K(z)\,dz\\
				&+\frac{1}{2}\int_{B_{r}(0)\setminus B_{\eps}(0)}(\varphi(x)-\varphi(x-z))\,K(-z)\,dz\\
				&=\frac{1}{2}\int_{B_{r}(0)\setminus B_{\eps}(0)}\left(2\varphi(x)-\varphi(x+z)-\varphi(x-z)\right)\,K(z)\,dz.
			\end{split}
		\end{equation}
		Using the Fundamental Theorem of Calculus, we see that
		\begin{equation*}
			\begin{split}
				|2\varphi(x)-\varphi(x+z)-\varphi(x-z)|&= \left|\int_{0}^{1}\nabla\varphi(x+tz)\cdot z-\nabla\varphi(x-tz)\cdot z\,dt\right|\\
				&\leq \int_{0}^{1}2\|D^2 \varphi\|_{L^{\infty}( B_{r}(x))} |z|^2 t\,dt\\
				&= \|D^2 \varphi\|_{L^{\infty}(B_{r}(x))} |z|^2.
			\end{split}
		\end{equation*}
		Substituting this into~\eqref{eq:step 1 in boundedness} and leveraging~\eqref{eq:comparison of k with FL} we obtain 
		\begin{equation}\label{eq:step 2 in boundedness}
			\begin{split}
				\left|\int_{B_r(x)\setminus B_{\eps}(x)}(\varphi(x)-\varphi(y))\,K(x-y)\,dy\right|&\leq \frac{1}{2}\|D^2 \varphi\|_{L^{\infty}(B_{r}(x))}\int_{B_{r}(0)\setminus B_{\eps}(0)}|z|^2K(z)\,dz\\
				&\leq\frac{1}{2}\|D^2 \varphi\|_{L^{\infty}(B_{r}(x))}\Lambda\int_{B_{r}(0)\setminus B_{\eps}(0)}|z|^{-N-2s+2}\,dz\\
				&\leq C_{5}\|D^2 \varphi\|_{L^{\infty}(B_{r}(x))}\left(r^{2-2s}-\eps^{2-2s}\right),
			\end{split}
		\end{equation}
		where $C_{5}$ is a positive constant depending only on $N$, $s$, and $\Lambda$.\\
		Moreover, the assumption that $\varphi\in L^{1}_{s}(\RR^N)$ yields
		\begin{equation}\label{eq:step 3 in boundedness}
			\begin{split}
				&\left|\int_{\RR^N\setminus B_{r}(x)}(\varphi(x)-\varphi(y))\,K(x-y)\,dy\right|\\
				&\leq \int_{\RR^N\setminus B_{r}(x)}\left|\varphi(x)-\varphi(y)\right|\,K(x-y)\,dy \\
				&\leq \Lambda|\varphi(x)|\int_{\RR^N\setminus B_{r}(x)}|x-y|^{-N-2s}\,dy+ \Lambda\int_{\RR^N\setminus B_{r}(x)}\frac{|\varphi(y)|}{|x-y|^{N+2s}}\,dy\\
				&=\Lambda|\varphi(x)|\int_{\RR^N\setminus B_{r}(0)}|z|^{-N-2s}\,dz + \Lambda\int_{\RR^N\setminus B_{r}(x)}\frac{|\varphi(y)|}{|x-y|^{N+2s}}\,dy\\
				&\leq C_{6}r^{-2s}\Lambda |\varphi(x)|+\Lambda \int_{\RR^N\setminus B_{r}(x)}\frac{1+|y|^{N+2s}}{|x-y|^{N+2s}}\frac{|\varphi(y)|}{1+|y|^{N+2s}}\,dy,
			\end{split}
		\end{equation}
		where $C_{6}$ is a positive constant depending only on $N$ and $s$.\\
		Noting that, for all~$y\in \RR^N\setminus B_{r}(x)$,
		\begin{equation*}
			\frac{1+|y|^{N+2s}}{|x-y|^{N+2s}}\leq C_{7}\left(1+|x|^{N+2s}\right),
		\end{equation*}
		where, $C_{7} = \max\left\{2^{N+2s+1}+\frac{3}{r^{N+2s}}, \frac{2^{N+2s}}{r^{N+2s}}\right\}$, we see that
		\begin{equation}\label{eq:step 4 in boundedness}
			\begin{split}
				&\left|\int_{\RR^N\setminus B_{r}(x)}(\varphi(x)-\varphi(y))\,K(x-y)\,dy\right|\\
				&\leq C_{6}r^{-2s}\Lambda |\varphi(x)|+ C_{7}\Lambda(1+|x|^{N+2s})\int_{\RR^N\setminus B_{r}(x)}\frac{|\varphi(y)|}{1+|y|^{N+2s}}\,dy\\
				&\leq C_{8} \left(|\varphi(x)|r^{-2s}+\|\varphi\|_{L^{1}_{s}(\RR^N)} \left(1+|x|^{N+2s}\right)\right),
			\end{split}
		\end{equation}
		where $C_{8} = \Lambda\max\{C_{6},C_{7}\}$.
		
		Combining \eqref{eq:step 2 in boundedness} and \eqref{eq:step 4 in boundedness} and recalling \eqref{def:nonlocal operator}, we conclude that, for any $x\in \RR^N$ and $r>\eps>0$,
		\begin{equation}\label{eq:step 5 in boundedness}
			\begin{split} 
				&\left|\int_{\RR^N\setminus B_{\eps}(x)}(\varphi(x)-\varphi(y))\,K(x-y)\,dy\right|\\
				&\leq C_{9} \left(\|D^2 \varphi\|_{L^{\infty}(B_{r}(x))}r^{2-2s}+ |\varphi(x)|r^{-2s}+\|\varphi\|_{L^{1}_{s}(\RR^N)} \left(1+|x|^{N+2s}\right)\right),
			\end{split}
		\end{equation}
		where $C_{9} =\max\{C_{5},C_{8}\}$.
		From this, the claim in~\eqref{eq:bound on nonlocal operator I} follows immediately.
		
		Now, further assume that $\varphi$ is compactly supported. 
		Given $x\in \RR^N$, we use \eqref{eq:step 2 in boundedness} to obtain, for any $r>\eps>0$, that
		\begin{equation}\label{eq:step 1 in decay}
			\begin{split} 
				&\left|\int_{\RR^N\setminus B_{\eps}(x)}\left(\varphi(x)-\varphi(y)\right)K(x-y)\,dy\right|\\
				&\leq	
				\left|\int_{B_{r}(x)\setminus B_{\eps}(x)}\left(\varphi(x)-\varphi(y)\right)K(x-y)\,dy\right|+ \left|\int_{\RR^N\setminus B_{r}(x)}\left(\varphi(x)-\varphi(y)\right)K(x-y)\,dy\right|\\
				&\leq C_{5}\|D^2 \varphi\|_{L^{\infty}(B_{r}(x))} r^{2-2s}+2\Lambda\|\varphi\|_{L^{\infty}(\RR^N)}\int_{\RR^N\setminus B_{r}(0)}|z|^{-N-2s}\,dz\\
				&\leq C_{10}\left(\|D^2 \varphi\|_{L^{\infty}(\RR^N)}r^2 + \|\varphi\|_{L^{\infty}(\RR^N)}\right)r^{-2s},
			\end{split}
		\end{equation}
		where $C_{5}$ is the constant that appears in \eqref{eq:step 2 in boundedness} and $C_{10}$ is a constant depending on $N$, $s$, $\Lambda$.\\
		If $x\in \RR^N$ is such that
		$$
		|x|> 2\max_{y\in \text{supp}(\varphi)} |y|,
		$$
		it follows that 
		$$
		|x-y|^{-N-2s}\leq 2^{N+2s}|x|^{-N-2s},\,\,\text{ for all }\,\, y\in \text{supp}(\varphi),
		$$
		and therefore
		\begin{equation}\label{eq:step 2 in decay}
			\begin{split}
				\left|\mathcal{L}_K \varphi(x)\right|&= \left|\int_{\text{supp}(\varphi)}\varphi(y)\,K(x-y)\,dy\right|\leq \Lambda\int_{\text{supp}(\varphi)}|\varphi(y)||x-y|^{-N-2s}\,dy\\
				&\leq 2^{N+2s}\Lambda\|\varphi\|_{L^{1}(\RR^N)} |x|^{-N-2s}.
			\end{split}
		\end{equation}
		
		Combining this and \eqref{eq:step 1 in decay}, the desired claim in~\eqref{eq:bound on nonlocal operator II} follows.
	\end{proof}
	\begin{proof}[Proof of Lemma~\ref{lem:distributional ineq}]
		
		We start by noting that since $u\in L^{1}_{s}(\RR^N;\RR^M)$, we have $u_{i}-\sqrt{k_n}\in~L^{1}_{s}(\RR^N)$, for every $i=1,...,M$. Furthermore, using Lemma~\ref{lem:integration by parts for nonlocal operator}, it is easy to see that
		$$
		\int_{\RR^N}(u_i-\sqrt{k_n})\mathcal{L}_K\varphi(x)\,dx = \int_{\RR^N} u_i(x)\mathcal{L}_K\varphi(x)\,dx,
		$$
		so that $\mathcal{L}_{K}(u_i-\sqrt{k_n})=\mathcal{L}_{K}u_i = -u_i P_{n}'(|u|^2)\in L^{1}_{\text{loc}}(\RR^N)$ (in the distributional sense). 
		Using Proposition~\ref{prop:Nonlocal Kato inequality} and \eqref{Main Problem} we see, for any nonnegative $\varphi\in \mathcal{D}(\RR^N)$, that
		\begin{equation}\label{eq:step 1 in distributional ineq}
			\begin{split}
				\int_{\RR^N} \left(u_i(x)-\sqrt{k_n}\right)^+\mathcal{L}_K\varphi(x)\,dx\leq -\int_{\RR^N}\text{sign}^+(u_i(x)-\sqrt{k_n}) \varphi(x) u_i(x)P_{n}'(|u(x)|^2)\,dx
			\end{split}
		\end{equation}
		As noted in \cite{MR1659256}, the derivative of the polynomial $P_n$ satisfies the relation
		$$
		P_n'(t) = Q_{n}(t)\prod_{i=1}^{n}(t-k_i),
		$$
		where 
		$$
		Q_n(t) = n\prod_{l=1}^{n-1}(t-r_l), 
		$$
		for some $r_l \in (k_l,k_{l+1})$ ($l=1,..., n-1$).
		
		With this in mind, the estimate in~\eqref{eq:step 1 in distributional ineq} yields
		\begin{equation}\label{eq:step 2 distributional ineq}
			\begin{split}
				&\int_{\RR^N}\left(u_i(x)-\sqrt{k_n}\right)^+\mathcal{L}_K\varphi(x)\,dx\\
				&\leq -\int_{\RR^N}\text{sign}^+(u_i(x)-\sqrt{k_n})\varphi(x) u_i(x)  \displaystyle\prod_{i=1}^{n}(|u(x)|^2-k_i)Q_{n}(|u(x)|^2)\,dx\\
				&\leq -\int_{\RR^N}\text{sign}^+(u_i(x)-\sqrt{k_n})\varphi(x) u_i(x)  \displaystyle\prod_{i=1}^{n}(|u_i(x)|^2-k_i)Q_{n}(|u_i(x)|^2)\,dx,
			\end{split}
		\end{equation}
		where the last inequality follows from the fact that $\displaystyle\prod_{i=1}^{n}(t-k_i)$ and $Q_n(t)$ are increasing for $t\geq k_n$.\\
		Defining 
		$$
		S(t):= \begin{cases}
			t(t+\sqrt{k_n})\displaystyle\prod_{j=1}^{n-1}(t^2-k_j)Q_n(t^2),\quad \text{if }\, t>\sqrt{k_n}\\
			2k_n\displaystyle\prod_{j=1}^{n-1}(k_n-k_j)Q_n(k_n), \quad \text{if }\, t\leq\sqrt{k_n},
		\end{cases}
		$$
		we thus obtain
		\begin{equation}\label{eq:step 3 in distributional ineq}
			\begin{split}
				\int_{\RR^N}\left(u_i(x)-\sqrt{k_n}\right)^+\mathcal{L}_K\varphi(x)\,dx &\leq -\int_{\RR^N}\varphi(x)S(|u_i(x)|)(u_i(x)-\sqrt{k_n})^+ \,dx.
			\end{split}
		\end{equation}
		Hence, taking $0<\alpha \leq S(\sqrt{k_n})$,
		\begin{equation}\label{eq:step 4 in distributional ineq}
			\begin{split}
				\int_{\RR^N}\left(u_i(x)-\sqrt{k_n}\right)^+\mathcal{L}_K\varphi(x)\,dx\leq -\alpha\int_{\RR^N}\varphi(x)(u_i(x)-\sqrt{k_n})^+ \,dx.
			\end{split}
		\end{equation}
		
		Noticing that $-u$ is also a distributional solution of \eqref{Main Problem}, we can apply the above procedure to $-u$ to conclude that
		\begin{equation}\label{eq:step 5 in distributional ineq}
			\begin{split}
				\int_{\RR^N}\left(-u_i(x)-\sqrt{k_n}\right)^+\mathcal{L}_K\varphi(x)\,dx\leq -\alpha\int_{\RR^N}\varphi(x)(-u_i(x)-\sqrt{k_n})^+ \,dx,
			\end{split}
		\end{equation}
		for every nonnegative $\varphi\in \mathcal{D}(\RR^N)$.
		
		Finally, since $(|u_i|-\sqrt{k_n})^+ = (-u_i-\sqrt{k_n})^+ + (u_i-\sqrt{k_n})^+$, the result follows from \eqref{eq:step 4 in distributional ineq} and \eqref{eq:step 5 in distributional ineq}.
	\end{proof}
	\section*{Acknowledgements} 
	Serena Dipierro, Giorgio Poggesi, and Enrico Valdinoci are members of the Australian Mathematical Society (AustMS).\\
	Serena Dipierro is supported by the Australian Future Fellowship FT230100333 ``New perspectives on nonlocal equations''.\\ 
	\\
	João Gonçalves da Silva is supported by a School of Mathematical Sciences Scholarship at Adelaide University.
	\\
	Giorgio Poggesi and João Gonçalves da Silva are supported by the Australian Research Council (ARC) Discovery Early Career Researcher Award (DECRA) DE230100954 ``Partial Differential Equations: geometric aspects and applications''.
	%%%
	%Giorgio Poggesi is also supported by the Australian Academy of Science %J G Russell Award. 
	%%%
	\\
	Enrico Valdinoci is supported by the Australian Laureate Fellowship FL190100081 ``Minimal surfaces, free boundaries and partial differential equations''.
	
	%Serena Dipierro is supported by the Australian Research Council
	%Future Fellowship FT230100333
	%``New perspectives on nonlocal equations''.
	%Jo{\~a}o Gon\c{c}alves da Silva and
	%Giorgio Poggesi are supported by the Australian Research Council Discovery Early Career Researcher Award (DECRA) DE230100954 ``Partial Differential Equations: geometric aspects and applications''.
	%Jo{\~a}o Gon\c{c}alves da Silva is supported by a Scholarship for International Research Fees at The University of Western Australia.
	%Enrico Valdinoci is supported by the Australian Laureate Fellowship FL190100081 ``Minimal surfaces, free boundaries and partial differential equations''.
	
%	\section*{Disclosure Statement}
	%The authors declare that they have no conflict of interest.
	
	\section*{Data Availability}
	No data was used for the research described in the article.

\end{document}